\documentclass[12pt]{article}
\usepackage{latexsym,amsmath,amssymb}
\usepackage[usenames]{color}

\newfont{\bb}{msbm10 at 12pt}

\def\r{\hbox{\bb R}}

\def\t{\hbox{\bb T}}

\def\ii{\hbox{\bb I}}
\def\c{\hbox{\cal C}}

\def\s{\hbox{\bb S}}

\def\pt{\frac{\partial\ }{\partial t}}

\def\amb{\mathcal{M}}

\newcommand{\norm}[1]{\left\Vert #1 \right\Vert}
\newcommand{\abs}[1]{\left\vert #1 \right\vert}
\newcommand{\set}[1]{\left\{#1\right\}}

\newcommand{\eps}{\varepsilon}

\usepackage[latin1]{inputenc}
\usepackage {amsmath}
\usepackage {amsthm}
\usepackage{times}
\usepackage{amscd}
\usepackage{epsf}

\numberwithin{equation} {section}

\begin{document}

\theoremstyle{plain}\newtheorem{lemma}{Lemma}[section]
\theoremstyle{plain}\newtheorem{definition}{Definition}[section]
\theoremstyle{plain}\newtheorem{theorem}{Theorem}[section]
\theoremstyle{plain}\newtheorem{proposition}{Proposition}[section]
\theoremstyle{plain}\newtheorem{remark}{Remark}[section]
\theoremstyle{plain}\newtheorem{corollary}{Corollary}[section]

\begin{center}
\rule{15cm}{1.5pt} \vspace{.6cm}

{\Large \bf Rigidity of stable cylinders in three-manifolds} \vspace{0.4cm}

\vspace{0.5cm}

{\large Jos$\acute{\text{e}}$ M. Espinar$\,^\ast$} \\
\vspace{0.3cm} \rule{15cm}{1.5pt}
\end{center}

\vspace{.5cm}

\noindent $\mbox{}^{\ast}$  Departamento de Geometría y Topología, Universidad de Granada, 18071 Granada, Spain; \\
e-mail: jespinar@ugr.es\vspace{.3cm}

\begin{abstract}
In this paper we show how the existence of a certain stable cylinder determines (locally) the ambient manifold where it is immersed. This cylinder has to verify a {\it bifurcation phenomena}, we make this explicit in the introduction. In particular, the existence of such a stable cylinder implies that the ambient manifold has infinite volume. 
\end{abstract}

\vspace{.3cm}


\vspace{.2cm}


\section{Introduction}

A stable compact domain $\Sigma$ on a minimal surface in a Riemannian three-manifold
$\amb$, is one whose area can not be decreased up to second order by a variation of
the domain leaving the boundary fixed. Stable oriented domains $\Sigma$ are
characterized by the \emph{stability inequality} for normal variations $\psi N$
\cite{SY}

$$ \int _{\Sigma} \psi ^2 |A|^2 + \int _{\Sigma} \psi ^2 {\rm Ric}_{\amb} (N,N)
\leq \int_{\Sigma} |\nabla \psi|^2  $$for all compactly supported functions $\psi
\in H^{1,2}_0 (\Sigma)$. Here $|A|^2$ denotes the the square of the length of the
second fundamental form of $\Sigma$, ${\rm Ric}_{\amb} (N,N)$ is the Ricci
curvature of $\amb$ in the direction of the normal $N$ to $\Sigma$ and
$\nabla $ is the gradient w.r.t. the induced metric.

One writes the stability inequality in the form
$$ \left.\frac{d^2}{dt^2}\right\vert _{t=0}{\rm Area}(\Sigma (t))=
- \int _{\Sigma} \psi L \psi \geq 0 ,$$where $L$ is the linearized operator of the
mean curvature
$$ L = \Delta + |A|^2 + {\rm Ric}_{\amb} .$$

In terms of $L$, stability means that $-L$ is nonnegative, i.e., all its eigenvalues
are non-negative. $\Sigma$ is said to have finite index if $-L$ has only finitely
many negative eigenvalues.

From the Gauss Equation, one can write the stability operator as $L = \Delta - K + V$, where $\Delta$ and $K$ are the Laplacian and Gauss curvature associated to the metric $g$ respectively, and $V := 1/2|A|^2 + S$, where $S$ denotes the scalar curvature associated to the metric $g$.

The index form of these kind of operators is
\begin{equation*}
I(f)=\int _{\Sigma } \left\{ \norm{\nabla f}^2 - V f^2+  K f^2 \right\}
\end{equation*}where $\nabla $ and $\| \cdot \|$ are the gradient and norm associated
to the metric $g$. Thus, if $\Sigma$ is stable, we have
$$\int_{\Sigma} f L f = -I(f) \leq 0 ,$$or equivalently
\begin{equation}\label{varcharac}
\int _ \Sigma f ^2 (1/2 |A|^2 + S) \leq \int _{\Sigma} \set{\norm{\nabla f}^2 + K f ^2} .
\end{equation}

In a seminar paper \cite{FCS}, D. Fischer-Colbrie and R. Schoen proved: 

\begin{quote}
{\bf Theorem A:} \emph{Let $\amb$ be a complete oriented three-manifold of non-negative scalar curvature. Let $\Sigma$ be an oriented complete stable minimal surface in $\amb$. If $\Sigma$ is noncompact, conformally equivalent to the cylinder and the absolute total curvature of $\Sigma$ is finite, then $\Sigma$ is flat and totally geodesic.}
\end{quote}

And they state \cite[Remark 2]{FCS}: \emph{We feel that the assumption of finite
total curvature should not be essential in proving that the cylinder is flat and
totally geodesic.}

Recently, this question was partially answered in \cite{ER} under the assumption that the positive part of the Gaussian curvature is integrable, i.e. $K^+:={\rm max}\set{0,K}\in L^1 (\Sigma)$, and totally answered by M. Reiris \cite{Re}, he proved:

\begin{quote}
{\bf Theorem B:} \emph{Let $\amb$ be a complete oriented three-manifold of non-negative scalar curvature. Let $\Sigma$ be an oriented complete stable minimal surface in $\amb$ diffeomorphic to the cylinder, then $\Sigma$ is flat and totally geodesic.}
\end{quote}

Besides, Bray, Brendle and Neves \cite{BBN} were able of determining the structure of a three-manifold $\amb$ under the assumption of the existence of an area minimizing two-sphere. Specifically, they proved:

\begin{quote}
{\bf Theorem C:} {\it Let $\amb$ be a compact three-manifold with $\pi _2 ( \amb ) \neq 0$. Denote by $\mathcal{F}$ the set of all smooth maps $f: \s ^2 \to \amb$ which represent a non-trivial element of $\pi _2 (\amb)$. Set
$$ \mathcal{A}(\amb) := {\rm inf}\set{{\rm area}(f(\s^2))\, : \, \, f \in \mathcal{F}} .$$\newline
Then, 
$$ \mathcal{A}(\amb)  {\rm inf}_{\amb} R \leq 8 \pi ,$$where $R$ denotes the scalar curvature of $\amb$. Moreover, if the equality holds, then the universal cover of $\amb$ is isometric to the standard cylinder $\s ^2 \times \r$ up to scaling.}
\end{quote}

In this paper, we will go further. We will see how the existence of a stable cylinder verifying a {\it bifurcation phenomena} determines the ambient manifold $\amb$. First, let us make clear what we mean by {\it bifurcation phenomena}: 

\begin{definition}\label{def1}
We say that a complete minimal surface $\Sigma \subset \amb $ bifurcates if there exist $\delta > 0$ and a smooth map $ u : \Sigma \times (-\delta , \delta) \to \r  $ so that 
\begin{itemize}
\item For each $p \in \Sigma$, we have $u(x, 0)=0$ and $\pt |_{t=0}u(p,t) =1$. Moreover, $u(p,t) \geq 0$ if $t > 0$ and $u(p,t) \leq 0$ if $t <0$.
\item For each $t \in (-\delta, \delta)$, the surface 
$$ \Sigma _t  := \set{ {\rm exp}_p (u(p,t) N (p)) \, : \, \, p \in \Sigma },$$is a complete minimal surface. Here, ${\rm exp}$ denotes the exponential map in $\amb$.
\end{itemize}
\end{definition}

Now, we can state:

\begin{theorem}\label{t2}
Let $\amb $ be a complete oriented Riemannian three-manifold with nonnegative scalar curvature. Assume there exists $\Sigma \subset \amb $ a complete stable minimal surface conformally equivalent to a cylinder that bifurcates. Then, $\Sigma$ is flat, totally geodesic and $S$ vanishes along $\Sigma$. Moreover, there exists an open set $\mathcal U \subset \amb$  so that $\mathcal U$  is locally isometric to $\c \times (-\delta , \delta)$, where $\c $ denotes the standard cylinder $\s^1 \times \r$. Also, if any complete stable cylinder in $\amb$ bifurcate for an uniform $\delta >0$, then $\amb$ is locally isometric either to $\s^1 \times \r ^2$ or $\t^2 \times \r$ (here $\t^2$ is the flat tori).
\end{theorem}

We should point out the condition that $\Sigma $ bifurcates is necessary. In fact, one can construct the following example: Let $\c (-l , l)$ be the right cylinder of height $2l$ and radius $1$ endowed with the flat metric. Close it up with two spherical caps $S_i $, $i=1,2$ (one on the top and another on the bottom). Now, smooth the surface $\amb ^2 = \c (-l,l) \cup S_1 \cup S_2$ so that it is flat on $\c(-l + \eps , l - \eps)$, for some $\eps >0$, and has nonnegative Gaussian curvature. 

Consider the three-manifold $\amb ^3 = \amb ^2 \times \r $. One can see that, if we take a closed geodesic $\gamma (t) \subset \c (-l+\eps , l - \eps)$ , $t \in (-l+\eps , l - \eps)$, the surface $\Sigma (t) := \gamma (t) \times \r $ is a complete stable minimal cylinder in $ \amb $ that bifurcates, but, when we reach $t= l -\eps$, this property it might disappear (it could bifurcate as constant mean curvature surfaces at one side, but not minimal).  

One interesting consequence of Theorem \ref{t2} is the following:
\begin{corollary}\label{c1}
Let $\amb $ be a complete oriented Riemannian three-manifold with nonnegative scalar curvature. Assume there exists $\Sigma \subset \amb $ a complete stable minimal surface conformally equivalent to a cylinder that bifurcates. Then, 
$$ {\rm Vol}(\amb) = + \infty . $$
\end{corollary}

Actually, the above conclusion (that is, the above Corollary \ref{c1}) is also valid when the cylinder bifurcates only at one side.


\section{Preliminaries}

We denote by $\amb$ a complete connected orientable Riemannian three-manifold, with Riemannian metric $g$. Moreover, throughout this work, we will assume that its scalar curvature is nonnegative, i.e., $S\geq 0$. $\Sigma \subset \amb$ will be assumed to be connected  and oriented. 

We denote by $N$ the unit normal vector field along $\Sigma$. Let $p_0 \in \Sigma$ be a point of the surface and $D(p_0,s)$, for $s>0$, denote the geodesic disk centered at $p_0$ of radius $s$. We assume that $\overline{D(p_0 ,s)} \cap \partial \Sigma = \emptyset$. Moreover, let $r$ be the radial distance of a point $p$ in $D(p_0, s)$ to $p_0$. We write $D(s)=D(p_0 ,s)$.

We also denote
\begin{eqnarray*}
l(s) &=& {\rm Length}(\partial D(s)) \\
a(s) &=& {\rm Area}(D(s))\\
K(s) &=& \int _{D(s)} K \\
\chi (s)&=& \text{Euler characteristic of } D(s).
\end{eqnarray*}




Let $\Sigma \subset \amb$ be a stable minimal surface diffeomorphic to the cylinder, then, from Theorem B \cite{Re}, $\Sigma$ is flat and totally geodesic. We will give a (more general) proof of this result in the abstract setting of Schr\"{o}dinger-type operators:

\begin{lemma}\label{l1}
Let $\Sigma $ be a complete Riemannian surface. Let $L = \Delta + V - a K$ be a differential operator on $\Sigma$ acting on compactly supported $f \in H^{1,2}_0 (\Sigma)$, where $a >1/4$ is constant, $V \geq 0$, $\Delta$ and $K$ are the Laplacian and Gauss curvature associated to the metric $g$ respectively. 

Assume that $\Sigma$ is homeomorphic to the cylinder and $-L$ is non-negative. Then, $V\equiv 0$ and $K \equiv 0$, therefore, 
$$ {\rm Ker} L := \set{1} ,$$i.e., its kernel is the constant functions. Here, $L$ denotes the Jacobi operator.  
\end{lemma} 
\begin{proof}
Set $b \geq 1$ and let us consider the radial function 
$$f(r):=\left\{ \begin{matrix}  (1-r/s)^b & r \leq s \\ 0 & r> s \end{matrix}\right. ,$$where $r$ denotes the radial distance from a point $p_0\in \Sigma$. Then, from \cite[Lemma 3.1]{ER} (see also \cite{MPR}), we have
\begin{equation*}
\int_ {D(s)} (1-r/s)^{2b} V \leq 2a \pi G(s)  + \frac{b(b(1-4a)+2a)}{s^2}\int _{0}^s(1-r/s)^{2b-2} l(r),
\end{equation*}where 
\begin{equation*}
G(s):= - \int _0 ^s (f(r)^2)' \chi (r) .
\end{equation*}

Therefore, since $a>1/4$, we can find $b \geq 1$ so that $b (1-4a)+2a \leq 0$. So
\begin{equation*}
\int_ {D(s)} (1-r/s)^{2b} V \leq 2a \pi G(s) .
\end{equation*}

\begin{itemize}
\item {\bf Step 1:} $V$ vanishes identically on $\Sigma$. 

Suppose there exists a point $p_0 \in \Sigma $ so that $V(p_0) >0$. From now on, we fix the point $p_0$. Then, there exists $\epsilon > 0$ so that $V(q)\geq \delta $ for all $q \in D(\epsilon)=D(p_0,\epsilon)$. Since $\Sigma$ is topollogically a cylinder, there exists $s_0 >0$ so that for all $s>s_0$ we have $\chi (s) \leq 0$ (see \cite[Lemma 1.4]{Ca}). 

Now, from the above considerations, there exists $\beta >0 $ so that 
$$  0 < \beta \leq  2 a \pi G(s) .$$

But, following \cite{ER}, we can see that 
\begin{equation*}
\begin{split}
G(s) & = - \int _0 ^s (f(r)^2)' \chi (r) = - \int _0 ^{s_0} (f(r)^2)' \chi (r)- \int
_{s_0} ^s (f(r)^2)' \chi (r)\\
 & \leq - \int _0 ^{s_0} (f(r)^2)'  =  -\left( f(s_0) ^2 - f(0)^2\right) = - f(s_0)^2 +1 \\
 &= -\left( 1-s_0/s\right)^{2b} +1 ,
\end{split}
\end{equation*}since $- \int _{s_0} ^s (f(r)^2)' \chi (r) \geq 0$. Therefore,
$$ G(s) \leq 1- (1-s_0/s)^{2b} \to 0 , \text{ as } s \to + \infty ,$$ which is a contradiction. Thus, $V$ vanishes identically along $\Sigma$.

\item {\bf Step 2:} $K$ vanishes identically on $\Sigma$. In particular, $\Sigma$ is parabolic.

First, note that $L:= \Delta - a K$. From \cite{FC}, there is a smooth positive function $u$ on $\Sigma$ such that $ Lu =0$. Set $\alpha := 1/a$. Then, from \cite{MPR} (following ideas of \cite{FC}), the conformal metric $\tilde{ds^2}:=u^{2\alpha} ds^2$, where $ds^2$ is the metric on $\Sigma$, is complete and its Gaussian curvature $\tilde K$ of is non-negative, i.e. $\tilde K \geq 0$. 

On the one hand, the respective Gaussian curvatures are related by
$$ \alpha  \Delta \ln u = K - \tilde K u ^{2\alpha}.  $$

On the other hand, since $\Sigma$ is topologically a cylinder, the Cohn-Vossen inequality says
$$ \int _\Sigma  \tilde K \leq  0 ,$$that is, $\tilde K $ vanishes identically. 	

Thus, $K = \alpha \Delta \ln u$. From this last equation, we get: 
$$ a K = \frac{1}{u} \Delta u - \frac{\abs{\nabla u}^2}{u^2}, $$that is,
$$ \frac{\abs{\nabla u}^2}{u} = \Delta u - a K u=0 . $$

This last equation implies that $u$ is constant, and since $u$ satisfies $Lu =0$, we have that $K $ vanishes identically on $\Sigma$. In particular, $\Sigma$ is parabolic (see \cite[Lemma 5]{KO})

\end{itemize}

This implies that the Jacobi operator becomes $L:= \Delta$, and so the constant functions are in the kernel. But, since $\Sigma $ is parabolic, such a kernel has dimension one (see \cite{MaPRo}), therefore
$${\rm Ker} L := \set{1} . $$
\end{proof}

Set $\c := \s ^1 \times \r $ the flat cylinder, then we can parametrize $\Sigma$ as the isometric immersion $\psi _0: \c \to \amb $ where $\Sigma := \psi _0 (\c)$. Also, set $N_0 : \c \to N\Sigma$ the unit normal vector field along $\Sigma$. 

Assume $\Sigma$ bifurcates (see Definition \ref{def1}), then there exist $\delta >0 $ and a smooth map $u : \c  \times (-\delta , \delta) \to \r$ so that the surface $\Sigma _t := \psi _t ( \c)$, $\psi _t : \c \to \amb$ where
$$ \psi _t  (p):= {\rm exp}_{\psi_0(p)} (u(p,t)N_0(p)) \, , \, \, p \in \c ,$$is a complete minimal surface. 

For each $t \in (-\delta , \delta)$, the lapse function $\rho _t : \Sigma \to \r$ is defined by
$$ \rho _t (p) = g\left( N_t (p), \pt \psi _t (p)\right).$$

Clearly, $\rho _0 (p) =1$ for all $p \in \c$. Also, the lapse function satisfies the Jacobi equation

\begin{equation}\label{rhoJacobi}
\Delta _t \rho _t + ({\rm Ric}(N_t)+|A_t|^2)\rho _t = 0,
\end{equation}since $\psi _t (\c)$ is minimal for all $|t|<\delta$.

\begin{lemma}\label{l3}
There exists $0< \delta ' <\delta $ such that $\Sigma _t$ is a stable minimal surface for each $t \in (-\delta, \delta)$. Thus, $\Sigma_t$ is flat, totally geodesic and $S $ vanishes along $\Sigma _t$ for each $t \in (-\delta, \delta)$.
\end{lemma}
\begin{proof}
First, note that, the lapse function is not negative for all $|t| < \delta$ and therefore, by \eqref{rhoJacobi} and the Maximum Principle, either $\rho _t$ vanishes identically or $\rho _t >0$ for each $|t| < \delta$.

So, since 
$$ \rho _t \to \rho _0 \equiv 1  , \, {\rm as } \, t \to 0,$$thus, we can find a uniform constant $0 < \delta ' < \delta $ such that  $\rho _t > 0$ for all $|t| \leq  \delta '$.

Therefore, $\rho _t$, $|t|\leq \delta '$, is a positive function solving the Jacobi equation. This implies that $\Sigma _t$ is stable for all $|t|\leq \delta '$ (see \cite{FC}). 

The last assertion follows from Lemma \ref{l1} and $\Sigma _t$ be stable.

\end{proof}

\section{Proof of Theorem \ref{t2}}

%

From Definition \ref{def1} and Lemma \ref{l3}, there exists $\delta >0$ so that $\Sigma _t$ is a complete minimal stable surface, which is flat, totally geodesic and $S=0$ along $\Sigma _t$, for each $|t|<\delta$. 

Now, we follows ideas of \cite{BBN}. Since ${\rm Ric}(N_t)+ |A_t|^2 \equiv 0$ and $H(t)=0$ for each $|t|<\delta$, from \eqref{rhoJacobi} and $\Sigma _t$ being parabolic, we obtain that $\rho _t$ is constant. Thus, since $\Sigma _t$ is totally geodesic, 

$$\begin{matrix}
Y : & \c \times (-\delta , \delta)&  \to& \amb \\
 & (p,t) & \to & Y(p,t):=N_t(p)
 \end{matrix}$$ is parallel. Also, the flow of $N_t$ is a unit speed geodesic flow (see \cite{M}). Moreover, the map

$$\begin{matrix}
\Phi : & \Sigma \times (-\delta , \delta )&  \to& \amb \\
 & (p,t) & \to & \Phi(p,t):={\rm exp}_{\psi _0(p)} (t \, N(p))
 \end{matrix}$$is a local isometry onto $\mathcal U = \bigcup _{|t|<\delta} \Sigma _t$. Therefore, $\Phi$ is a diffeomorphism onto $\mathcal U$, which implies that $Y : \c \times (-\delta , \delta) \to \mathcal U $ is a globally defined unit Killing vector field. This implies that $\mathcal U$ is locally isometric to $\c \times (-\delta , \delta)$.

Now, assume that any stable minimal complete cylinder bifurcates for an uniform $\delta >0$. Then, we can start with a complete stable minimal cylinder $\Sigma _0$ that bifurcates, and then by the above considerations, $\Sigma _t$, for each $|t|<\delta$, is complete, flat, totally geodesic and $S$ vanishes along $\Sigma _t$. Moreover, $\Sigma _t$ is strongly stable for each $|t|< \delta $. Note that $\Sigma _ \delta$ is a strongly stable minimal surfaces conformally equivalent to a cylinder, since it is limit of strongly stable minimal surfaces $\Sigma _t$ which are flat and totally geodesic, then $\Sigma _\delta $ is totally geodesic, flat and $S =0 $ along $\Sigma _\delta$. Therefore, by Definition \ref{def1} and Lemma \ref{l3}, there exists $\delta >0 $ so that $\Sigma _t$, $- \delta < t < 2 \delta$, is flat, totally geodesic and $S $ vanishes along $\Sigma _t$. Continuing this argument, $\Sigma _t$ is flat, totally geodesic and $S$ vanishes along $\Sigma _t$ for each $t \in \ii $, where $\ii = \r$ or $\ii= \s ^1$.

As we did above, since ${\rm Ric}(N_t)+ |A_t|^2 \equiv 0$ and $H(t)=0$ for each $t \in \ii$, from \eqref{rhoJacobi} and $\Sigma _t$ being parabolic, we obtain that $\rho _t$ is constant. Thus, since $\Sigma _t$ is totally geodesic, 

$$\begin{matrix}
Y : & \c \times \ii&  \to& \amb \\
 & (p,t) & \to & Y(p,t):=N_t(p)
 \end{matrix}$$ is parallel, where $\ii=\r$ or $\ii = \s ^1$. Also, the flow of $N_t$ is a unit speed geodesic flow (see \cite{M}). Moreover, the map

$$\begin{matrix}
\Phi : & \Sigma \times \ii&  \to& \amb \\
 & (p,t) & \to & \Phi(p,t):={\rm exp}_{\psi _0(p)} (t \, N(p))
 \end{matrix}$$is a local isometry, which implies that it is a covering map. Therefore, 
$\Phi$ is a diffeomorphism, which implies that $Y : \c \times \ii \to \amb $ is a globally defined unit Killing vector field. This implies that $\amb$ is locally isometric either to $\s^1 \times \r ^2$ or $\t ^2 \times \r$ (here $\t ^2$ denotes the flat tori). 

\begin{center}
{\bf Acknowledge}
\end{center}

The author is partially supported by Spanish MEC-FEDER Grant MTM2010-19821, and Regional J. Andalucia Grants P06-FQM-01642 and FQM325.

\end{document}